\documentclass{article}
\usepackage{amsmath,amssymb,bbm,theorem}

\newcommand{\assign}{:=}

\numberwithin{equation}{section}

\newcommand{\tmop}[1]{\ensuremath{\operatorname{#1}}}

\newcommand{\tmstrong}[1]{\textbf{#1}}

\newenvironment{proof}{\noindent\textbf{Proof\ }}{\hspace*{\fill}$\Box$\medskip}

\newtheorem{definition}{Definition}
\newtheorem{lemma}{Lemma}

{\theorembodyfont{\rmfamily}}
\newtheorem{theorem}{Theorem}

\newcommand{\XXint}[3]{{\setbox}0=\text{\ensuremath{#1 #2 #3 \int}}
{\vcenter{\text{\ensuremath{#2 #3}}}}{\kern}-.5{\tmwd}0}

\newcommand{\opn}[2]{\newcommand{\1}{\}} {\opn}{\Rm{Rm}} {\opn}{\Ric{Ric}}
{\opn}{\Rc{Rc}} {\opn}{\Scal{Sc}} {\opn}{\Tr{Tr}} {\opn}{\Trac{Tr}}
{\opn}detdet {\opn}{\diam{diam}} {\opn}{\dist{dist}} {\opn}{\Im}Im
{\opn}{\div}div {\opn}{\Ker{Ker}} {\opn}expexp {\opn}{\Vol{Vol}}
{\opn}{\exph{exph}} {\opn}{\Herm{Herm}} {\opn}{\End{End}} {\opn}{\Hess{Hess}}
{\opn}{\Vol{Vol}}}

\newcommand{\contract}{{\kern}-1.5pt{\vrule} width6.0pt height0.4pt depth0pt
{\vrule} width0.4pt height4.0pt depth0pt}
\newcommand{\retract}{{\kern}-1.5pt{\vrule} width0.4pt height4.0pt depth0pt
{\vrule} width6.0pt height0.4pt depth0pt}
\newcommand{\Openbox}{{\leavevmode} {\hfil}{\vrule} width{\boxrulethickness}
{\vbox} to{\Openboxwidth{{\advance}{\Openboxwidth} -2{\boxrulethickness}
{\hrule} height {\boxrulethickness} width{\Openboxwidth}{\vfil} {\hrule}
height{\boxrulethickness}}}{\vrule} width{\boxrulethickness}{\hfil} }

\begin{document}

\title{Exact Fourier inversion formula over manifolds}\author{\\
{\tmstrong{NEFTON PALI}}}\date{}\maketitle
\begin{abstract}
  We show an exact (i.e. no smooth error terms) Fourier inversion type formula
  for differential operators over Riemannian manifolds.
  This provides a coordinate free approach for the theory of pseudo-differential
  operators.
\end{abstract}

Intrinsic symbolic calculus was pioneered by Widom \cite{Wid1, Wid2}, and has received
contributions from Fulling-Kennedy, Safarov, Sharafutdinov, \cite{Fu-Ke, Saf, Sha1, Sha2}. In this works the authors
produce Fourier type inversion formulas over Riemannian manifolds. Our concern
is that all this formulas presents (smooth) error therms. Such therms can be
very tedious when one has to consider intrinsic type computations over such
manifolds. On the other hand exact (i.e. no smooth error terms) inversion
formulas allow to simplify the proof of the Atiyah-Singer index theorem for
the Dirac operator on a spin manifold (see \cite{Get}).

We show now our version of the Fourier inversion formula over manifolds. 
We introduce first our set-up and notations.

Let $(X, g)$ be a smooth Riemannian manifold and let $(E, h_E)$, $(F, h_F)$ be
two smooth hermitian vector bundles over $X$. We assume that $E$ is of the
complexification of a vector bundle of the type $S^{\lambda} T_X
\otimes_{\mathbbm{R}} S^{\mu} T^{\ast}_X$, where $S^{\lambda}$ denotes a Schur
power indexed by $\lambda$ and $h_E$ is the sesquilinear extension of the
metric induced by $g$. We will denote by $\nabla_g$ the induced connection
over the bundles $(T^{\ast}_X)^{\otimes p} \otimes_{\mathbbm{R}} E$.

\begin{definition}
  A differential operator $A$ of order $p$ over $(X, g)$ acting on the
  sections of bundle $(E, h_E)$ with values in the sections of $(F, h_F)$ is a
  linear map $A : C^{\infty} (X, E) \longrightarrow C^{\infty} (X, F)$ of the
  type
  \begin{eqnarray*}
    A & = & \sum_{r = 0}^p A_r \nabla^{p - r}_g,
  \end{eqnarray*}
  with $A_r \in C^{\infty} (X, S^{p - r} T_X \otimes_{\mathbbm{R}} E^{\ast}
  \otimes_{\mathbbm{R}} F)$ and $E^{\ast} \assign \tmop{Hom}_{\mathbbm{R}} (E,
  \mathbbm{R})$. The total symbol of $A$ is the fibre map over $X$
  \begin{eqnarray*}
    a & \assign & \sum_{r = 0}^p a_r \in C^{\infty} (T^{\ast}_X, E^{\ast}
    \otimes_{\mathbbm{R}} F),
  \end{eqnarray*}
  with $a_r (\lambda) \assign (2 \pi i \lambda)^{\otimes (p - r)} \neg A_{r
  \mid \pi_X (\lambda)} \in E_{\pi_X (\lambda)}^{\ast} \otimes_{\mathbbm{R}}
  F_{\pi_X (\lambda)}$ and $\pi_X : T^{\ast}_X \longrightarrow X$.
\end{definition}

With this notations we can state our Fourier inversion formula over Riemannian manifolds.

\begin{theorem}
  For all points $x \in X$ let $D_x \subset T_{X, x}$ be the connected
  component of $0_x$ such that the map $\exp_{g, x} : D_x \longrightarrow X
  \smallsetminus \tmop{Cutlocus} (g, x)$ is a diffeomorphism and let $\tau_x^E
  : E_{\mid_{D_x}} \longrightarrow E_x$ be the parallel transport map of the
  fibers of $E$ along the geodesics raising from the point $x$. Assume that $A
  : C^{\infty} (X, E) \longrightarrow C^{\infty} (X, F)$ is a differential
  operator over $X$. Then for all $u \in C^{\infty} (X, E)$
  hold the Fourier type inversion formula
  \begin{eqnarray*}
    A u (x) & = & \int_{\lambda \in T^{\ast}_{X, x}} d V_{g_x^{\ast}}
    (\lambda) a (\lambda) \int_{\xi \in D_x} \tau_x^E \cdot u \circ \exp_{g,
    x} (\xi) e^{- 2 \pi i \lambda \cdot \xi} \chi_x (\xi) d V_{g_x} (\xi),
  \end{eqnarray*}
  where $\chi_x (\xi) \assign \chi (| \xi |_{g_x} / \varepsilon_x)$, with
  $\chi : [0, + \infty) \longrightarrow [0, 1]$ a fixed smooth function such
  that $\chi (t) = 1$ for $t \in [0, 1]$, $\chi (t) = 0$ for all $t \geqslant
  2$ and $\varepsilon_x \in \mathbbm{R}_{> 0}$ such that $\{| \xi |_{g_x}
  \leqslant 2 \varepsilon_x \} \subset D_x$.
\end{theorem}

\begin{proof}
We observe first of all that a basic fact about the Fourier transform in
  $\mathbbm{R}^n$ implies directly that the function
  \begin{eqnarray*}
    \lambda \in T^{\ast}_{X, x} \longmapsto \widehat{u}_x (\lambda) & \assign
    & \int_{\xi \in D_x} \tau_x^E \cdot u \circ \exp_{g, x} (\xi) e^{- 2 \pi i
    \lambda \cdot \xi} \chi_x (\xi) d V_{g_x} (\xi),
  \end{eqnarray*}
  belongs to the Schwartz space $\mathcal{S} (T_{X, x}^{\ast}, E_x)$.
  Therefore the integral in the statement make sense. Furthermore the Fourier
  inversion formula shows that for any function $U \in \mathcal{S} (T_{X, x},
  E_x)$ hold the identity
  \begin{equation}
    \label{Four-inv} U (0) = \int_{\lambda \in T^{\ast}_{X, x}} d
    V_{g_x^{\ast}} (\lambda) \int_{\xi \in T_{X, x}} U (\xi) e^{- 2 \pi i
    \lambda \cdot \xi} d V_{g_x} (\xi) .
  \end{equation}
  Let now $U$ be the extension by $0$ over $T_{X, x}$ of the function 
  $$\xi \in
  D_x \longmapsto \tau_x^E \cdot u \circ \exp_{g, x} (\xi) \chi_x (\xi).
  $$
  We
  deduce
  \begin{eqnarray*}
    u (x) & = & \int_{\lambda \in T^{\ast}_{X, x}} d V_{g_x^{\ast}} (\lambda)
    \int_{\xi \in D_x} \tau_x^E \cdot u \circ \exp_{g, x} (\xi) e^{- 2 \pi i
    \lambda \cdot \xi} \chi_x (\xi) d V_{g_x} (\xi),
  \end{eqnarray*}
since $U (0) = u (x)$. This shows the case $A = \tmop{Id}$. We show next the
  case of first order operators. Let $\eta \in C^{\infty} (X, T_X)$ and set
  for notations simplicity $\eta_x \assign \eta (x)$. We apply
  (\ref{Four-inv}) to the function $U$ obtained extending by $0$ the function
  \[ \xi \in D_x \longmapsto \left[ \eta_x . \left( \chi_x \tau_x^E \cdot u
     \circ \exp_{g, x} \right) \right] (\xi) . \]
  The identity (\ref{Four-inv}) implies
  \begin{eqnarray*}
    &  & \left[ \eta_x . \left( \tau_x^E \cdot u \circ \exp_{g, x} \right)
    \right] (0)\\
    &  & \\
    & = & \left[ \eta_x . \left( \chi_x \tau_x^E \cdot u \circ \exp_{g, x}
    \right) \right] (0)\\
    &  & \\
    & = & \int_{\lambda \in T^{\ast}_{X, x}} d V_{g_x^{\ast}} (\lambda)
    \int_{\xi \in D_x} \left[ \eta_x . \left( \chi_x \tau_x^E \cdot u \circ
    \exp_{g, x} \right) \right] (\xi) e^{- 2 \pi i \lambda \cdot \xi} d
    V_{g_x} (\xi) .
  \end{eqnarray*}
  Integrating by parts we infer
  \begin{eqnarray*}
    &  & \left[ \eta_x . \left( \tau_x^E \cdot u \circ \exp_{g, x} \right)
    \right] (0)\\
    &  & \\
    & = & \int_{\lambda \in T^{\ast}_{X, x}} d V_{g_x^{\ast}} (\lambda) 2 \pi
    i \lambda \cdot \eta_x \int_{\xi \in D_x} \tau_x^E \cdot u \circ \exp_{g,
    x} (\xi) e^{- 2 \pi i \lambda \cdot \xi} \chi_x (\xi) d V_{g_x} (\xi) .
  \end{eqnarray*}
  On the other hand we observe the identities
  \begin{eqnarray*}
    \left[ \eta_x . \left( \tau_x^E \cdot u \circ \exp_{g, x} \right) \right]
    (0) & = & \frac{d}{d t} \vphantom{dt}_{|_{t = 0}}  \tau_x^E \cdot u \circ \exp_{g, x}
    (t \eta_x) = \nabla_{g, \eta} u (x) .
  \end{eqnarray*}
  We deduce the first order Fourier type inversion formula
  \begin{eqnarray*}
   && \nabla_{g, \eta} u (x) 
    \\
    \\
    & = & \int_{\lambda \in T^{\ast}_{X, x}} d
    V_{g_x^{\ast}} (\lambda) 2 \pi i \lambda \cdot \eta_x \int_{\xi \in D_x}
    \tau_x^E \cdot u \circ \exp_{g, x} (\xi) e^{- 2 \pi i \lambda \cdot \xi}
    \chi_x (\xi) d V_{g_x} (\xi) .
  \end{eqnarray*}
  We show now the following basic arbitrary order Fourier type inversion
  formula. Let $\eta_1, \ldots, \eta_p \in C^{\infty} (X, T_X)$ and set $\eta
  \assign \sum_{\sigma \in S_p} \eta_{\sigma_1} \otimes \cdots \otimes
  \eta_{\sigma_p}$. For notations simplicity we set $\eta_{k, x} \assign
  \eta_k (x)$, for $k = 1, \ldots, p$ and $\eta_x \assign \eta (x)$. Then
  \begin{eqnarray*}
    & &\nabla^p_{g, \eta} u (x) 
    \\
    \\
    & = & \int_{\lambda \in T^{\ast}_{X, x}} d
    V_{g_x^{\ast}} (\lambda)  (2 \pi i \lambda)^p \neg \eta_x \int_{\xi \in
    D_x} \tau_x^E \cdot u \circ \exp_{g, x} (\xi) e^{- 2 \pi i \lambda \cdot
    \xi} \chi_x (\xi) d V_{g_x} (\xi) .
  \end{eqnarray*}
  Notice the identity $(2 \pi i \lambda)^p \neg \eta_x = p! (2 \pi i)^p
  (\lambda \cdot \eta_{1, x}) \cdots (\lambda \cdot \eta_{p, x})$. In order to
  show this inversion formula we consider the function
  \[ \xi \in D_x \longmapsto \left[ \eta_{1, x} \ldots \eta_{p, x} . \left(
     \chi_x \tau_x^E \cdot u \circ \exp_{g, x} \right) \right] (\xi) . \]
  Using the identity (\ref{Four-inv}) we obtain
  \begin{eqnarray*}
    &  & \left[ \eta_{1, x} \ldots \eta_{p, x} . \left( \tau_x^E \cdot u
    \circ \exp_{g, x} \right) \right] (0)\\
    &  & \\
    & = & \left[ \eta_{1, x} \ldots \eta_{p, x} . \left( \chi_x \tau_x^E
    \cdot u \circ \exp_{g, x} \right) \right] (0)\\
    &  & \\
    & = & \int_{\lambda \in T^{\ast}_{X, x}} d V_{g_x^{\ast}} (\lambda)
    \int_{\xi \in D_x} \left[ \eta_{1, x} \ldots \eta_{p, x} . \left( \chi_x
    \tau_x^E \cdot u \circ \exp_{g, x} \right) \right] (\xi) e^{- 2 \pi i
    \lambda \cdot \xi} d V_{g_x} (\xi) .
  \end{eqnarray*}
  A multiple integration by parts yields
  \begin{eqnarray*}
  &  & \left[ \eta_{1, x} \ldots \eta_{p, x} . \left( \tau_x^E \cdot u
    \circ \exp_{g, x} \right) \right] (0)\\
    &  & \\
    & = & \int_{\lambda \in T^{\ast}_{X, x}} d V_{g_x^{\ast}} (\lambda)  (2
    \pi i)^p (\lambda \cdot \eta_{1, x}) \cdots (\lambda \cdot \eta_{p, x})
     \times
    \\
    \\
    &\times&
    \int_{\xi
    \in D_x} \tau_x^E \cdot u \circ \exp_{g, x} (\xi) e^{- 2 \pi i \lambda
    \cdot \xi} \chi_x (\xi) d V_{g_x} (\xi) .
  \end{eqnarray*}
  Then the required inversion formula follows from the identity
  \begin{eqnarray*}
  & &\left[ \eta_{1, x} \ldots \eta_{p, x} . \left( \tau_x^E \cdot u \circ
    \exp_{g, x} \right) \right] (0) 
    \\
    \\& = & \frac{\partial^p}{\partial t_1
    \cdots \partial t_p} \vphantom{dt}_{|_{t_1, \ldots,t_p = 0}} \tau_x^E \cdot u \circ
    \exp_{g, x} (t_1 \eta_{1, x} + \cdots + t_p \eta_{p, x}),
  \end{eqnarray*}
  and from the differential identity
  \begin{equation}
    \label{cov-id} \nabla^p_{g, \eta_x} u (x)  =  p! \frac{\partial^p}{\partial t_1 \cdots
    \partial t_p} \vphantom{dt}_{|_{t_1, \ldots,t_p = 0}} \tau_x^E \cdot u \circ
    \exp_{g, x} (t_1 \eta_{1, x} + \cdots + t_p \eta_{p, x}),
  \end{equation}
  that we will prove next. (Compare with lemma 7.5 in \cite{Sha1}). At this point the general statement in the theorem
  follows immediately.
\end{proof}

{\tmstrong{Proof of the identity (\ref{cov-id})}}. 
We show first a semi-group property of the geodesic flow. Over a Riemannian
manifold $\left( X, g \right)$ we consider a vector field $\xi$ and we denote
by $e \left( \xi \right) : X \longrightarrow X$ the smooth map defined by the
rule $e \left( \xi \right)_x \assign \exp_g \left( x, \xi_x \right) \equiv
\exp_{g, x} \left( \xi_x \right)$.

\begin{lemma}
  \label{geodFlow}Let $\gamma_t \assign \exp_{g, x} \left( t \eta \right)$,
  $\eta \in T_{X, x} \smallsetminus \{0\}$ be a geodesic and let $\xi$ be the
  vector field over a small neighborhood of $x$ inside $\tmop{Im} \gamma$
  defined as $\xi_{\gamma_t} = \dot{\gamma}_t$. Then hold the semi group
  property $e \left( t \xi \right) \circ e \left( s \xi \right)_x = e \left(
  \left( t + s \right) \xi \right)_x$, for all $t, s \in \left( - \varepsilon,
  \varepsilon \right)$, for some sufficiently small $\varepsilon > 0$.
\end{lemma}

\begin{proof}
  We fix $s$. Let $\beta_t \assign e \left( t \xi \right) \circ e \left( s \xi
  \right)_x = \exp_g \left( \gamma_s, t \dot{\gamma}_s \right)$ and $\theta_t
  \assign \gamma_{t + s}$. The conclusion will follow from the identity
  $\beta_t = \theta_t$ that we show next. We notice that the geodesic $\beta$
  satisfies the initial conditions $\beta_0 = \gamma_s$, $\dot{\beta}_0 =
  \dot{\gamma}_s$. But the curve $\theta$ is also a geodesic which satisfies
  the same initial conditions. Indeed we observe the identities $\theta_0 =
  \gamma_s$ and 
  $$\dot{\theta}_{\tau} = \frac{d}{dt}
    \vphantom{dt}_{|_{t = \tau}} 
  \gamma_{t + s} = \dot{\gamma}_{\tau + s}.
  $$ 
  The later implies $\dot{\theta}_0
  = \dot{\gamma}_s$ and $\nabla_{\dot{\theta}_{\tau}}  \dot{\theta}_{\tau} =
  \nabla_{\dot{\gamma}_{\tau + s}}  \dot{\gamma}_{\tau + s} \equiv 0$. Then
  the identity $\beta_t = \theta_t$ follows from the uniqueness of the
  solutions of ODE. 
\end{proof}

\begin{lemma}
  The symmetrized multi-covariant derivative
  \begin{eqnarray*}
    \hat{\nabla}^p_g & \assign & \frac{1}{p!}  \sum_{\sigma \in S_p}
    \nabla^p_{g, \sigma},
  \end{eqnarray*}
  satisfies the formula
  \begin{eqnarray*}
    \hat{\nabla}^p_g u \left( x \right) & = & d^p_0 \left( \tau_x^E \cdot u
    \circ \exp_{g, x} \right),
  \end{eqnarray*}
  for any smooth section $u \in C^{\infty} \left( X, E \right)$ and any point
  $x \in X$. In more explicit terms
  \begin{eqnarray*}
    \hat{\nabla}^p_{g, \eta_1, \ldots, \eta_p} u (x) & = &
    \frac{\partial^p}{\partial t_1 \cdots \partial t_p} \vphantom{dt}_{|_{t_1, \ldots,
    t_p = 0}}  \tau_x^E \cdot u \circ \exp_{g, x} (t_1 \eta_1 + \cdots + t_p
    \eta_p),
  \end{eqnarray*}
  for any vectors $\eta_j \in T_{X, x}$.
\end{lemma}

\begin{proof}
  The fact that both sides are symmetric multi-linear maps over $T_{X, x}$
  implies that the statement to prove is equivalent to the identity
  \begin{equation}
    \label{partcCovId} \nabla_{g, \eta^{\otimes p}}^p u \left( x \right) =
    \frac{d^p}{d t^p} \vphantom{dt}_{|_{t = 0}}  \tau_x^E \cdot u \circ \exp_{g, x}
    \left( t \eta \right) .
  \end{equation}
  (This simplification of the problem was suggested to us by Pierre Pansu). We
  show (\ref{partcCovId}) by induction on $p$. The case $p = 1$ is a
  reformulation of the notion of covariant derivative. We assume now
  (\ref{partcCovId}) true for $p$. With the notations of lemma \ref{geodFlow}
  hold the identity $\nabla_{\xi} \xi \equiv 0$. This implies
  \begin{eqnarray*}
    \nabla_{g, \xi^{\otimes p + 1}}^{p + 1} u \left( x \right) & = &
    \nabla_{g, \xi}  \left( \nabla_{g, \xi^{\otimes p}}^p u \right)  \left( x
    \right)\\
    &  & \\
    & = & \frac{d}{d s}\vphantom{ds}_{|_{s = 0}}  \tau_x^E \cdot \left( \nabla_{g,
    \xi^{\otimes p}}^p u \right) \circ e \left( s \xi \right)_x\\
    &  & \\
    & = & \frac{d}{d s}\vphantom{ds}_{|_{s = 0}}   \tau_x^E \cdot \frac{d^p}{d t^p}\vphantom{dt}_{|_{t = 0}} 
     \tau_{e \left( s \xi \right)_x}^E \cdot u \circ e \left( t
    \xi \right) \circ e \left( s \xi \right)_x\\
    &  & \\
    & = & \frac{d}{d s}\vphantom{ds}_{|_{s = 0}}   \frac{d^p}{d t^p}\vphantom{dt}_{|_{t = 0}}  
    \tau_x^E \cdot \tau_{e \left( s \xi \right)_x}^E \cdot u \circ e \left(
    \left( t + s \right) \xi \right)_x,
  \end{eqnarray*}
  thanks to lemma \ref{geodFlow}. Simplifying further we obtain
  \begin{eqnarray*}
    \nabla_{g, \eta^{\otimes p + 1}}^{p + 1} u \left( x \right) & = &
    \nabla_{g, \xi^{\otimes p + 1}}^{p + 1} u \left( x \right)\\
    &  & \\
    & = & \frac{d}{d s} \vphantom{ds}_{|_{s = 0}}  \frac{d^p}{d t^p}\vphantom{dt}_{|_{t = 0}}  
    \tau_x^E \cdot u \circ \exp_{g, x} \left( \left( t + s \right) \eta
    \right)\\
    &  & \\
    & = & \frac{d^{p + 1}}{d t^{p + 1}} \vphantom{dt}_{|_{t = 0}}  \tau_x^E \cdot u
    \circ \exp_{g, x} \left( t \eta \right),
  \end{eqnarray*}
  and thus the required conclusion of the induction.
\end{proof}

\vspace{1cm}
\noindent
Nefton Pali
\\
Universit\'{e} Paris Sud, D\'epartement de Math\'ematiques 
\\
B\^{a}timent 425 F91405 Orsay, France
\\
E-mail: \textit{nefton.pali@math.u-psud.fr}

\end{document}